\documentclass[11pt]{article}

\usepackage{color}
\usepackage{amssymb,amscd,amsmath,amsfonts,amsthm}

\theoremstyle{plain}

\newtheorem{theo}{Theorem}[section]
\newtheorem{lem}[theo]{Lemma}
\newtheorem*{lem*}{Lemma}

\newtheorem{prop}[theo]{Proposition}

\theoremstyle{remark}

\newtheorem{rem}[theo]{Remark}

\theoremstyle{definition}

\newcommand{\ga}{\gamma}
\newcommand{\Si}{\Sigma}

\newcommand{\al}{\alpha}

\newcommand{\si}{\sigma}
\newcommand{\om}{\omega}
\newcommand{\Ga}{\Gamma}
\newcommand{\be}{\beta}
\newcommand{\la}{\lambda} 
\newcommand{\Om}{\Omega}
 
\newcommand{\ep}{\epsilon}

\newcommand{\Hilb}{{\mathcal{H}}}

\newcommand{\C}{\mathbb C}
\newcommand{\R}{\mathbb R}

\newcommand{\Z}{\mathbb Z}
\newcommand{\T}{\mathbb T}

\newcommand{\su}{\operatorname{SU}_2}

\newcommand{\mo}{{\mathcal{M}}}  
\newcommand{\Ad}{\operatorname{Ad}} 
\newcommand{\tr}{\operatorname{Tr}}

\newcommand{\Ci}{{\mathcal{C}}^{\infty}}  

\newcommand{\CS}{\operatorname{CS}}  

\newcommand{\alt}{\operatorname{alt}} 

\newcommand{\ab}{\operatorname{ab}} 
\newcommand{\ir}{\operatorname{irr}} 
\newcommand{\id}{\operatorname{Id}}
\newcommand{\MS}{\operatorname{MS}}
\newcommand{\cl}{\operatorname{cl}}

\title{Torus knot state asymptotics}

\author{L. Charles\footnote{Institut de
    Math{\'e}matiques de Jussieu (UMR 7586), Universit{\'e} Pierre et
    Marie Curie -- Paris 6, Paris, F-75005 France.}}

\begin{document}

\maketitle

\begin{abstract}
The state of a knot is defined in the realm of Chern-Simons topological quantum field theory as a holomorphic section on the $\su$-character manifold of the peripheral torus. We compute the asymptotics of the torus knot states in terms of the Alexander polynomial, the Reidemeister torsion and the Chern-Simons invariant. We also prove that the microsupport of the torus knot state is included in the character manifold of the knot exterior. As a corollary we deduce the Witten asymptotics conjecture for the Dehn filling of the torus knots and asymptotic expansions for the colored Jones polynomials. 
\end{abstract}

\bibliographystyle{alpha}

\section{Introduction}

In the seminal paper \cite{Wi}, Witten developed a quantum field theory with Chern-Simons action, which allowed him to reinterpret Jones polynomials and to introduce new invariants of three dimensional manifolds. He showed that the semi-classical limit of this theory involves Chern-Simons invariant and Reidemeister torsion. 

In collaboration with Julien March{\'e} \cite{LJ1, LJ2}, we applied methods of semi-classical analysis to investigate this limit. We conjectured that the knot states are Lagrangian state supported by the character manifold of the knot exterior with as symbol the Reidemeister torsion and phase the Chern-Simons invariant. As a consequence the Witten-Reshetikhin-Turaev invariant of the Dehn filling of the knots have the asymptotic expansion predicted by Witten. We proved our conjecture for the figure eight knot in \cite{LJ1, LJ2}. In the present paper, we prove it for the torus knots. 

The knot state is computed from the colored Jones polynomial of the knot. The asymptotic behavior of the colored Jones polynomial of the torus knot has been considered in many works, including the Melvin-Morton paper \cite{melvin_morton},  the Volume conjecture of Kashaev \cite{kt} and its further development by Murakami \cite{Murakami}, Hikami \cite{hikami04}, and Dubois-Kashaev \cite{dk}, cf. \cite{hikami_murakami} for more references. 
Our results are rather different in that we do not consider directly the colored Jones polynomials, but the knot state. Let us be more precise.


Following \cite{retu}, associated to the Lie group $\su$ and to a positive integer $k$ is a topological quantum field theory (tqft), that is a functor from a cobordism category to the category of Hermitian vector spaces. Let $K$ be a knot in the three dimensional sphere. Denote by $E_K$ the complement of an open tubular neighborhood of $K$ and by  $\Si= \partial E_K$ the peripheral torus of $K$. The tqft associates to $\Sigma$ a $(k-1)$ dimensional Hermitian space $V_k(\Sigma)$ and to $E_K$ a vector $Z_k(E_K)\in V_k(\Sigma)$, that we call the knot state.
$V_k( \Si)$ admits an orthonormal basis $(e_\ell)$ such that the coefficient of the knot state are evaluations of the colored Jones polynomials $(J_\ell)$ of the knot,
$$Z_k ( E_K ) =  
 \sqrt{ \frac{2}{k}} \sin(\pi/k ) \sum_{\ell =1, \ldots, k-1} J_\ell (-e^{i\pi/2k}) e_{\ell}$$
Here the Jones polynomials are normalized in such a way that for the trivial knot $J_{\ell}(t) =   ( t^{2\ell} - t^{-2 \ell}) ( t^2 - t ^{-2})^{-1}$

Let $\mo(\Sigma)$ be the space of representations of the fundamental group $\pi ( \Si)$ in $\su$ up to conjugation. 
The space $V_k(\Sigma)$ is isomorphic to the geometric quantization $\Hilb_k^{\alt}$ at level $k$ of $\mo(\Sigma)$. The moduli space $\mo(\Sigma)$  is a complex orbifold with four singular points corresponding to the central representations. It it the base of two holomorphic line (orbi-)bundles, the Chern-Simons bundle $L_{\CS}$ and a half-form bundle $\delta$ respectively. The quantum space $\Hilb_k^{\alt}$ is the space of holomorphic sections of $L^k _{\CS} \otimes \delta$. It can be described very explicitly as a space of theta functions. In this paper we will identify $V_k( \Si)$ and $\Hilb_k^{\alt}$ using a special class of isomorphisms introduced in \cite{LJ1}. 

Let $\mo (E_K)$ be the moduli space of representations $\pi ( E_K) \rightarrow \su$ up to conjugation. $\mo (E_K)$ is the union of $\mo^{\ab }(E_K)$ and $\mo ^{\ir} (E_K)$ which consist respectively of abelian and irreducible representations. 
$\mo^{\ab }(E_K)$ is (homeomorphic to) a closed interval, its endpoints being central representations and its interior consisting in abelian non central representations. For the torus knot with parameter $(a,b)$, the closure of $\mo^{\ir} (E_K)$ in $\mo (E_K)$ is the disjoint union of  $(a-1) (b-1) /2$ closed intervals. The interior of each of these intervals consists of irreducible representations and the endpoints are abelian non central representations. So there are exactly $(a-1) (b-1)$ abelian representations which are limits of irreducible ones. Denote by $X_K$ the set of these representations.

Since the boundary of $E_K$ is the peripheral torus $\Si$, there is a natural map $r$ from $\mo ( E_K)$ to $\mo ( \Si)$. For torus knots, the restrictions of $r$ to $\mo^{\ab} (E_K)$ and to each component of the closure of $\mo^{\ir} (E_K)$ are embeddings. Nevertheless $r$ is not injective.

\begin{theo} \label{theo:intro}
Let $K$ be a torus knot with parameter $(a,b)$ and let $\tau$ be a non central representation in $\mo (\Si)$. Then if $\tau \notin r( \mo ( E_K))$,
$$ Z_k (E_K) ( \tau  ) = O(k^{-N}), \qquad \forall N $$
If $\tau \in r( \mo (E_K)) \setminus r(X_K)$, then we have an asymptotic expansion
\begin{gather*} 
   Z_k (E_K)  ( \tau) = \sum_{\rho \in r^{-1}( \al) \cap \mo^{\ir} (E_K)} \CS^k ( \rho)  \frac{k^{3/4 } \mu_k } { 4 \pi ^{3/4}} \sqrt{\T ( \rho)}   \\
  +   \frac{e^{-i \frac{\pi}{4}}}{\sqrt{2}} \sum_{\rho \in r^{-1}( \al) \cap \mo^{\ab} (E_K)}  \CS^k ( \rho)  \Biggl( \frac{\la_{\rho} - \bar{\la}_\rho}{\Delta_K(\la_{\rho}^2)} + \sum _{n \in \Z_{>0}} k^{-n} a_n(\rho)  \Biggr) \Om_{\la}(\tau)  +O(k^{-\infty}) 
\end{gather*}
where for any $\rho \in \mo( E_K)$ 
\begin{itemize} 
\item[-] $\CS ( \rho)$ is the Chern-Simons invariant of $\rho $,
\item[-] $\sqrt{ \T ( \rho)}$ is one of the four square roots of the Reidemeister torsion of $\rho$ if $\rho \in \mo^{\ir} ( E_K)$ and $\mu_k = \exp ( i \frac{\pi}{2k} ( - \frac{ a^2 + b^2}{ ab} + \frac{ab}{2} )) = 1 +O(k^{-1})$
\item[-]  $\la_{\rho}$ is an eigenvalue of $\rho ( \mu)$ with $\mu$ a meridian, $\Delta_{K} \in \Z [ t^{\pm 1} ]$ is the Alexander polynomial of $K$, $(a_n ( \rho), n>0)$ is a family of complex numbers, if $\rho \in \mo^{\ab} ( E_K)$.
\end{itemize}
\end{theo}

Let us comment the various ingredients. Chern-Simons invariants for three-dimensional oriented compact manifolds with boundary have been defined in \cite{ramadas_singer_weitsman}. In particular, any representation $\rho \in \mo (E_K)$ has a Chern-Simons invariant $\CS ( \rho) $ which is a unitary vector in the fiber of $L_{\CS} \rightarrow \mo (\Si)$ at $r ( \rho)$.

Any irreducible representation $\rho \in \mo^{\ir} (E_K)$ has a Reidemeister Torsion $\T ( \rho)$ which is a linear form of $H^1 (E_K, \Ad_{\rho} )$ well-defined up to sign. $ H^1 (E_K, \Ad_{\rho} )$ is naturally isomorphic to the tangent space of $\mo^{\ir} (E_K)$ at $\rho$. Any vector  in $T^*_{\rho} \mo^{\ir} (E_K) $ admits a square root well-defined up to sign in $\delta_{\rho}$. So the Reidemeister $\T ( \rho)$  has exactly four square root in  $\delta_{\rho}$. 

Our normalization for the Alexander polynomial is such that $\Delta_K(1) = 1$ and $\Delta_K( t) = \Delta_{K} ( t^{-1})$.  An abelian representation $\rho$ in $\mo ^{\ab} (E_K)$ belongs to $X_K$ if and only if the squares of the eigenvalues of $\rho ( \mu)$ are roots of $\Delta_K$. So the leading term of the asymptotic expansion is not defined precisely when $\tau \in r(X_K)$. 

Finally, $\Om_{\la} ( \tau)$ is a vector in the fiber $\delta_{\tau}$ of the half-form bundle, which depends (up to a power of $i$) on the choice of the isomorphism $V_k ( \Si) \simeq \Hilb_k^{\alt}$. It does not depend on the knot. 

We actually prove finer results which give the asymptotic expansion of the torus knot states at any representation $\tau \in \mo ( \Si) \setminus r(X_K)$ uniformly with respect to $\tau$, cf. Theorems \ref{theo:irreductible}, \ref{theo:abelian} and \ref{theo:abelien_irreductible}. This uniform asymptotic expansion contains the transition between $r( \mo (E_K))$ and the complementary set. As was proved in \cite{LJ2}, these results imply the Witten asymptotic expansion conjecture for Dehn fillings of the torus knot. For any three dimensional closed oriented manifold $M$, we denote by $Z_k(M)$ the Witten-Reshetikhin-Turaev invariant of $M$. 

\begin{theo} \label{theo:witten}
Let $M$ be a manifold obtained by Dehn surgery on a torus knot $K$. Assume that for any representation $\rho \in \mo (M)$ the restriction of $\rho$ to the knot exterior $E_K$ does not belong to $X_K$. Then

\begin{xalignat*}{2} 
 Z_k (M) &  = \sum_{\rho \in \mo (M)}  e^{i \frac{m (\rho) \pi }{4}} k^{n(\rho)} a(\rho,k) \CS ( \rho )^k  + O(k^{-\infty}) 
\end{xalignat*}
where for any $\rho \in \mo (M)$,  $m( \rho)$ is an integer, $n( \rho)  = 0$, $-1/2$ or $-3/2$ according to whether $\rho$ is irreducible, abelian non-central or central. Furthermore $(a ( \rho, k ))_k$ is a sequence which admits an asymptotic expansion of the form $a_0 ( \rho) + k^{-1} a_1 ( \rho ) + \ldots$. The leading coefficient is given by
$$ a_0 ( \rho) = \begin{cases} 2^{-1} \bigl( \T ( \rho) \bigr)^{1/2}  \text{ if $\rho$ is irreducible} \\ 
 2^{-1/2} \bigl( \T ( \rho) \bigr)^{1/2} \text{ if $\rho$ is abelian non-central} \\ 2^{1/2} \pi / p ^{3/2}    \text{ if $\rho$ is central.}\end{cases}$$
 where $(p,q)$ are the parameters of the surgery.  
\end{theo}

The condition that no representation of $M$ restricts to a representation in $X_K$ is equivalent to the following assertion: for any integers  $\ell$ and $m$ such that $m$ is not divisible by $a$ or $b$, the surgery coefficient satisfies   $$\frac{p}{q} \neq 2 ab \frac{\ell }{m }.$$ 
Here $(a,b)$ are the parameters of the torus knot. 
By Moser \cite{moser}, the Dehn filling of the torus knot are Seifert manifolds. The Witten asymptotic expansion conjecture has already been proved for many Seifert manifolds, cf. the discussion in \cite{LJ2}. 

As another corollary, we obtain an asymptotic expansion for $J_{m} ( - e^{i \pi /2k})$ in the large $m$ and $k$ limit, in the regime where $m/k$ is bounded.  When $|\frac{m}{2k}| < \frac{1}{2ab}$, this asymptotic expansion is an analytic version of the Melvin-Morton-Rozanski theorem, the leading order term being a function of the Alexander polynomial. When $ \frac{1}{2ab}<  \frac{m}{2k} < \frac{1}{2} - \frac{1}{2ab} $, the leading order term involves the Reidemeister torsion and Chern-Simons invariant of the irreducible representations of the  knot exterior, cf. Theorem \ref{sec:Jones_pol_das}. 

To complete this introduction, let us discuss our proof. First the colored Jones polynomials of the torus knot have been computed by Morton in \cite{morton}. Hikami proved that they satisfy some $q$-difference relations \cite{hikami}. From these relations, one deduces equations satisfied by the knot state on which our study is based. These equations involve two characteristic sets $A_0$ and $A_{ab}$ in the moduli space $\mo( \Si) \simeq \R^2 / \Z^2 \rtimes \Z_2$  
$$ A_0 = \{ [ 0, x] /\; x \in \R\}, \quad A_{ab} = \{ [-ab x , x ]/\; x \in \R \} \cup \{ [-ab x , x + \tfrac{1}{2ab} ]/\; x \in \R \}. $$  
The set $A_0$ is the image $r ( \mo^{\ab} (E_K))$ whereas $A_{ab}$ contains $r ( \mo ^{\ir} (E_K))$. We proved in \cite{LJ1} that the knot state is $O(k^{-\infty})$ on the complementary set of $ A_0 \cup A_{ab}$. We also obtained the asymptotic expansion of the knot state on $A_0 \setminus A_{ab}$, the remaining set $A_{ab}$ is treated in the present paper. 

$A_{ab} \setminus A_0$ is the union  of $ab$ open intervals $(I_j )_j$. On each of them, the knot state behaves as a Lagrangian state similarly to a WKB solution of a differential equation. The difficulty is to understand what happens when we pass from one interval to  another. It is clear from Theorem \ref{theo:intro} that something non trivial must occur, because the restriction of $r$ from $r^{-1} (I_j) $ to $I_j$  is a covering which degree depends on $j$. For some intervals $I_j$, it even happens that $r^{-1} (I_j)$ is empty so that the knot state is $O(k^{-\infty})$ on $I_j$. Fortunately, we can solve this problem by using elementary techniques. Essentially we work with an adapted basis where we can compute explicitly the knot state and extract easily its asymptotics. It is surprising that we can in this way determine explicitly the asymptotics expansion on the irreducible representations up to a $O(k^{-\infty})$. To the contrary, the proof in \cite{LJ1} based on Toeplitz operators, just gives the leading term of the asymptotics.
The study at the intersection $I_0 \cap I_{ab}$ is more difficult. We consider only the points $x \in I_0 \cap I_{ab}$ whose preimage $r^{-1}(x)$ consists of regular points. The other points, which form the set $r(X_K)$, are not considered in this article. 

The paper is organized as follows. Section \ref{sec:geom-quant-tori} is devoted to the geometric quantization of tori, including the Heisenberg group representation and its semi-classical limit. We introduce the knot sates and the vector spaces $V_k ( \Si)$ in section \ref{sec:knot-state}. The analytical results are proved in section \ref{sec:knot-state-analysis}, their topological interpretation is made in Section \ref{sec:topol-invar}. 

{\bf Acknowledgments.}
This work would not have been possible without the contribution of Julien March{\'e}. I  would like also to thank Fr{\'e}d{\'e}ric Faure for his help on numerical simulations and on the quantum mechanics of the torus.

\section{On the geometric quantization of tori} \label{sec:geom-quant-tori}
\subsection{Finite Heisenberg group representations} \label{sec:finite-heis-group}

Consider a real two dimensional symplectic vector space $(E, \om)$ with a compatible linear complex structure $j$. Let $(\delta, \varphi)$ be half-form line, that is $\delta$ is a complex line and $\varphi $ an isomorphism from $\delta^{\otimes 2}$ to the canonical line 
\begin{gather} \label{eq:canonical_line}
K_j= \{ \alpha \in E^* \otimes \C/ \al ( j \cdot ) = i \al \}.
\end{gather}
Let $\al \in \Om^1 (E, \C)$ be given by $\al_x(y) = \frac{1}{2} \om (x,y)$. Denote by $L$ the trivial complex line bundle over $E$ endowed with the connection
$d + \frac{1}{i} \al$ and with the unique compatible holomorphic structure. Consider the trivial holomorphic line bundle with fiber $\delta$ and base $E$ and denote it by $\delta$ again.

Let $k$ be a positive integer. The Heisenberg group at level $k$ is $E \times U(1)$ with the product
\begin{gather}  \label{eq:loi_Heisenberg}
 (x,u).(y,v) = \Bigl( x+y, uv \exp \Bigl( \frac{ik}{2} \om (x,y)\Bigr)\Bigr).
\end{gather}
The same formula where $(y,v) \in L^k\otimes \delta \simeq E \times \C $ defines an action of the
Heisenberg group on $L^k \otimes \delta$. For any $x\in E$, we denote by $T_x^*$ the pull-back by
the action of $(x, 1)$, explicitly
\begin{gather} \label{eq:pull_back} (T_x^* \Psi )(y) = \exp \Bigl( -
i\frac{k}{2} \om (x, y) \Bigr) \Psi ( x+ y)
\end{gather} for any section $\Psi$ of $L^k \otimes \delta$. This operator preserves the space of holomorphic sections.

Let $N$ be a positive integer. Let $R$ be a lattice of $E$ with a basis $(e, f)$ such that $\om ( e,f ) = 2\pi N$. Then $ R \times \{ 1
\}$ is a subgroup of the Heisenberg group. Define $\Hilb_k^R$ as the space of $R$-invariant holomorphic sections of $L^k \otimes \delta$. So a holomorphic section $\Psi$ belongs to $\Hilb_k^R$ if and only if $ T_x ^* \Psi = \Psi$ for any $x \in R$. 
$\Hilb^R_k$ is a finite dimensional space with dimension $Nk$. It has a natural scalar product given by 
\begin{gather} \label{eq:scalar_product} \langle \Psi_1 , \Psi_2
\rangle = \int_{ D} \langle \Psi_1 (x), \Psi_2 (x) \rangle_{\delta} \;
|\om | (x), \qquad \Psi_1, \Psi_2 \in \Hilb_k
\end{gather} where $D$ is any fundamental domain of $R$. Here $\delta$ is endowed with the metric making the isomorphism $\varphi : \delta^2 \rightarrow K_j$ an isometry. 

Let $U_{2Nk}$ be the group of $2Nk$-th root of unity. $\frac{1}{Nk} R \times U_{2Nk}$ is a subgroup of the Heisenberg group, that we call a finite Heisenberg group. It is contained in the commutator of $R \times \{ 1 \}$, so it acts on $\Hilb_k^R$. By Theorem 2.2 in \cite{LJ1}, $\Hilb_k^R$ has an orthonormal basis $(\Psi_{\ell} ; \ell \in \Z/ N k \Z)$ satisfying 
\begin{gather} \label{eq:rep_heis}
 T^*_{e / Nk} \Psi_{\ell} = e^{ \frac{2i \pi}{Nk} \ell} \Psi_\ell , \qquad T^*_{f/ Nk} \Psi_\ell = \Psi_{\ell +1}
\end{gather}
So the eigenvalues of $T^*_{e / Nk}: \Hilb_k \rightarrow \Hilb_k$ are the numbers $ e^{ \frac{2i \pi}{Nk} \ell}$, $\ell \in \Z / Nk \Z$, they are all simple. We can construct such a basis by choosing first a normalized $\Psi_0 \in \ker ( T^*_{e/ Nk} - \id)$ and defining the other vectors $\Psi_\ell$ with the second equation of (\ref{eq:rep_heis}). The only indeterminacy in the choice of this basis is the phase of $\Psi_0$.  In this paper we will often use the following normalization of the phase: 
\begin{gather} \label{eq:norm} 
\Psi_0(0) \in  \Om_{e} \R_{>0}.
\end{gather}
where $\Om_e \in \delta$ is such that $\varphi(\Om_e^2) ( e ) = 1$. 
We have that $$\Psi_\ell ( -x) = \Psi_{-\ell} (x).$$  
This property does not depend on the normalization.

\subsection{Asymptotical properties}

\subsubsection{Basis vectors as Lagrangian states} \label{sec:basis-lagrangian}

Consider a fixed basis $(e,f)$ of the lattice $R$. Choose $\Om_e$ satisfying $\varphi ( \Om_e^2 ) (e) =1$. Then for any $k$, there exists a unique orthonormal basis $(\Psi_{\ell ,k })$ of $\Hilb_k^R$ satisfying (\ref{eq:rep_heis}) and (\ref{eq:norm}). In this section we collect some asymptotical properties of the states $\Psi_{\ell,k}$ as $k$ tends to infinity.  

First the sequence $(\Psi_{0,k})$ concentrates on $\R e + \Z f $ in the following sense: for any $\delta \in (0,1)$, there exists a positive $C$ such that 
\begin{gather} \label{eq:estim_1_psi_0}
\bigl| \Psi_{0,k} ( x)  \bigl| \leqslant C e^{-k/C} ,
\end{gather}
 for all  $x \in \frac{1}{2}f + \frac{1}{2}[-\delta, \delta ]f + \R e$ modulo $R$ and every $k$. 

Second $\Psi_{0,k}$ has the following form on a neighborhood of $\R e$. Let $t_e$ be the unique holomorphic section of $L$ which restricts to the constant section equal to 1 on $\R e$. Then for any $\delta \in (0,1)$, there exists a positive $C$ such that 
 \begin{gather} \label{eq:estim_2_psi_0}
 \bigl| \Psi_{0,k} ( x) - \Bigl( \frac{k}{2\pi} \Bigr)^{1/4} t_e^k (x) \otimes \Om_{e}  \bigl| \leqslant C e^{-k/C} 
\end{gather}
for all  $x \in  [- \delta , \delta] f  + \R e$ and every $k$. In particular
\begin{gather}  \label{eq:psi_0_asymptot}
 \Psi_{0,k}( 0 ) = \Bigl( \frac{k}{2 \pi } \Bigr)^{ 1/4} \Om_e + O( k^{- \infty} ) 
\end{gather}
Equations (\ref{eq:estim_1_psi_0}) and (\ref{eq:estim_2_psi_0}) are the content of Proposition 3.2 of \cite{LJ1}.

Introduce the eigenvector 
$$ \Phi_k = \frac{1}{\sqrt{k N}} \sum _{ m \in \Z / k N \Z} \Psi_{m,k}$$ 
of $T^*_{f/ Nk}$ with eigenvalue 1.  One proves by applying Poisson summation formula (cf. proof of Theorem 2.3 in \cite{LJ1}) that 
\begin{gather} \label{eq:phi_asymptot}
 \Phi_k ( 0 ) =   e^{i \frac{\pi}{4} }  \Bigl( \frac{k}{2 \pi } \Bigr)^{ 1/4}   \Om_f + O( k^{- \infty} )
\end{gather}
where $\Om_f \in \delta$ is such that $\varphi(\Om_f ^2) ( f) = 1$. Then considering the basis $(f, -e)$ of $R$ instead of $(e,f)$ , one deduces from Equation (\ref{eq:estim_1_psi_0}) and (\ref{eq:estim_2_psi_0}) the asymptotic behavior of $(\Phi_{k})$. For any $\delta \in (0,1)$, there exists a positive $C$ such that
$$\bigl| \Phi_k ( x)  \bigl| \leqslant C e^{-k/C} ,$$
 for all  $x \in \frac{1}{2}e + \frac{1}{2}[-\delta, \delta ]e + \R f$ modulo $R$ and every $k$. Furthermore 
$$ \bigl| \Phi_k ( x) - e ^{i \frac{\pi}{4}} \Bigl( \frac{k}{2\pi} \Bigr)^{1/4} t_f^k (x) \otimes \Om_{f}  \bigl| \leqslant C e^{-k/C} $$ 
for all  $x \in  [- \delta , \delta] e  + \R f$ and every $k$. Here   $t_f$ is the unique holomorphic section of $L$ restricting to the constant section equal to 1 on $\R f$.

\subsubsection{Microsupport} 

 Consider a family $( \xi_k \in \Hilb_k^R; \; k \in \Z_{>0})$. We say $(\xi_k)$ is {\em admissible} if there exists a positive $C$ and an integer $M$ such that 
$$\bigl\|  \xi_k  \bigr\|_{\Hilb_k} \leqslant C k^{M} .$$ 
The {\em microsupport} of $(\xi_k)$ is the subset $\MS(\xi_k)$ of $E$ defined as follows: $x \notin \MS( \xi_k)$ if and only if there exists a  neighborhood $U$ of $x$ and a sequence $(C_M)$ of positive numbers such that for any $M$ and any $k$, 
$$  | \xi_k (y) | \leqslant C_M k^{-M}, \qquad \forall y \in U $$ 
$\MS (\xi_k)$  is a closed set of $E$, $R$-invariant. 

\begin{rem} \label{sec:micro_support_base}
By (\ref{eq:estim_1_psi_0}) and (\ref{eq:estim_2_psi_0}), the microsupport of $(\Psi_{0,k})$ is $\R e + \Z f$. More generally consider a family $(\xi_k\in \Hilb_k)$ of normalized vectors such that $$T_{e/Nk}^* \xi_k = e ^{ 2 i \pi \la_k} \xi_k$$ with $\la_k$ a sequence converging to $\la$. Then using that $ \xi_k = T^*_{\la_kf} \Psi_{0,k}$ up to a phase, we deduce from (\ref{eq:estim_1_psi_0}) and (\ref{eq:estim_2_psi_0}) that the microsupport of $(\xi_k)$ is $-\la f +\Z f +  \R e$. 
\qed \end{rem}

In the sequel we will deduce the microsupport of some state $(\xi_k)$  from the asymptotic behavior of the coefficients of $\xi_k$ in the basis $(\Psi_{\ell, k })$, by using the following proposition. 

\begin{prop} \label{prop:estim}
Let $\xi = ( \xi_k \in \Hilb_k^R; \; k \in \Z_{>0})$ be an admissible family. Let $I \subset \R$ be an open interval such that for any integer $\ell$
$$ - \ell/kN \in I   \quad \Rightarrow \quad  \langle \xi_k, \Psi_{k,\ell}\rangle = 0 $$
Then the microsupport of $\xi$ does not intersect the set $ I f + \R e$.
\end{prop}

\begin{proof} 
Write $\xi_k = \sum \al_{k,\ell} \Psi_{k, \ell}$. $(\xi_k)$ being admissible one has
\begin{gather} \label{eq:admissib}
| \al_k | \leqslant C' k^{M} 
\end{gather}
for some $M$ and positive $C'$. 
Denote by $(p, q)$ the linear coordinates of $E$ dual to $(e,f)$. Let $x \in I f + \R e$, so $q(x) \in I$. Choose $\delta \in (0,1)$ such that $q (x) + [ -1 + \delta) , 1 - \delta ] \subset I$. Let $C$ be a positive constant such that Equation (\ref{eq:estim_1_psi_0}) is satisfied. Since $\Psi_{\ell, k} = T^*_{\ell f/ kN } \Psi_{0,k}$, it follows from (\ref{eq:estim_1_psi_0}) that we have for any integer $n$ 
\begin{gather}  \label{eq:implication}
 \Bigl| q(y) +\frac{\ell}{kN} + \frac{1}{2} +n \Bigr|  \leqslant \frac{\delta}{2} \quad   \Rightarrow \quad \bigl| \Psi _{\ell, k } (y)   \bigr| \leqslant C e^{-k/C} 
\end{gather}
Let $y$ be such that  $| q(y) -   q(x) | <  \frac{1}{2}(1- \delta)$. If $\ell/Nk \in I$ modulo $\Z$, then $\al_{\ell, k} = 0$. Otherwise, one has for any $n$
$$\bigl| q (y)  + \frac{\ell}{Nk} + n \bigr| > \frac{1}{2} ( 1 -\delta)$$ 
which implies by (\ref{eq:implication}) that $\bigl| \Psi _{\ell, k } (y)   \bigr| \leqslant C e^{-k/C}$. 
 Using this with (\ref{eq:admissib}) and the fact that the dimension  of $\Hilb_{k}^R$ is $Nk$,  we obtain 
$$ | \Phi_k(y) | \leqslant (Nk)  C' k^{M} C e ^{-k/C} .
 $$ 
This concludes the proof. 
\end{proof}

\section{Knot state} \label{sec:knot-state}

\subsection{Definitions} \label{sec:def_knot_state}

Consider as previously a real two dimensional symplectic vector space $(E, \om)$ with a compatible linear complex structure $j$ and a half-form bundle $(\delta, \varphi)$. Let $(\lambda, \mu)$ be a basis of $E$ such that $\om ( \lambda,\mu) = 4\pi$. We denote by  $\Hilb_k$ the space of holomorphic sections of $L^k
\otimes \delta$ invariant with respect to the lattice $\la \Z \oplus \mu \Z$. Introduce the following endomorphisms of $\Hilb_k$
\begin{gather} \label{eq:def_M_L_q}
  M = T^*_{\mu/2k}, \quad L = T^*_{- \lambda / 2k} \quad \text{ and } \quad q = e^{i \frac{\pi}{k}}. 
\end{gather}
Let $\Om_\mu \in \delta$ satisfying $\varphi(\Om_\mu^2) ( \mu) = 1$.
As recalled in section \ref{sec:finite-heis-group}, $\Hilb_k$ has a unique orthonormal basis $( \xi_\ell, \; \ell \in \Z/ 2 k \Z )$ such that
\begin{gather}  \label{eq:rep_M_L}
 M \xi_\ell = q^{\ell } \xi_\ell , \qquad L \xi _\ell = \xi_{\ell - 1 }
\end{gather}
and $\xi_0(0) \in \R^+ \Om_{\mu}$. 

Let $a, b$ be relatively prime positive integers. For any integer $\ell$, we denote by $J_\ell$ the $\ell$-th Jones polynomial of the torus knot with parameter $(a, b)$. Here we normalize the Jones polynomials in such a way that the $\ell$-th  Jones polynomial of the trivial knot is $( t^{2 \ell } - t^{- 2 \ell})/(t^2 - t ^{-2})$. 
For any positive $\ell$, it is proved in \cite{morton} that  
\begin{gather} \label{eq:Pol_jones}
 J_\ell ( t) = \frac{ t ^{ab ( 1 - \ell^2)}}{ t^2 - t^{-2}} \sum _{r = - \frac{\ell -1}{2} }^{\frac{\ell -1}{2}} t^{4ab r^2} ( t^{ - 4 ( a + b ) r +2 } - t ^{-4 ( a - b ) r - 2 })
\end{gather}
We set $J_0 = 0$ and $J_{-\ell} = - J_{\ell}$ for any positive $\ell$. It is a property satisfied by any knot that $J_{\ell}(-e^{i \pi/2k}) = J_{\ell + 2k } ( - e^{i \pi/2k})$. 
We let $Z_k$ be the vector of $\Hilb_k$ 
\begin{gather}   \label{eq:def_Z_k}
Z_k = \frac{\sin(\pi/k )}{\sqrt{k}} \sum_{\ell \in \Z / 2k \Z} J_\ell (-e^{i\pi/2k}) \xi_\ell 
\end{gather}
and call it the knot state.

\subsection{Topological quantum field theory}

Let us present the knot state in the realm of topological quantum field theory. We work with the group $\su$ at a given level $k$ and we do not take into account the anomaly correction to simplify the exposition. To any closed oriented surface $\Si$ is associated a Hermitian vector space $V_k (\Si)$. To any oriented compact three dimensional manifold $N$ with a colored banded link $(L,c)$ is associated a vector $Z_k ( N, L,c) \in V_k ( \partial N)$. Here the color $c$ is a map from $\pi_0 (L)$ to $\{ 1, \ldots , k-1 \}$

Consider a knot $K$ in the three dimensional sphere with peripheral torus $\Si$. Let $N_K$ be a tubular neighborhood of the knot with boundary $\Si$. Thicken the knot $K$ to get an annulus $L$ whose boundary components are unlinked. Then the family $(Z_k ( N_K, L , \ell), \; \ell =1 , \ldots , k-1 )$ is a basis of $V_k (\Si)$. It follows from the definitions of TQFT that the vector $Z_k(E_K) \in V_k ( \Si)$ associated to the knot exterior is given in this basis by the formula
\begin{gather}\label{eq:def_knot_State_top} 
 Z_k ( E_K ) =  
 \sqrt{ \frac{2}{k}} \sin(\pi/k ) \sum_{\ell =1, \ldots, k-1} J_\ell (-e^{i\pi/2k}) Z_k ( N_K, L , \ell)  
\end{gather} 
We refer the reader to Section 2.1 and Section 4.2 of \cite{LJ1} for more details. 

Consider the vector space $E = H ^1 ( \Si, \R)$ with its lattice $R = H^1 ( \Si, \Z)$. Endow $E$ with the symplectic form given by $4 \pi$ times the intersection product. Let $\la$ and $\mu$ in $R$ be the (classes of a) longitude and meridian of the knot. Introduce a compatible linear complex structure with a half-form line and define the quantum space ${\Hilb_k}$ as in Section \ref{sec:def_knot_state}. Then choosing a vector $\Om_{\mu} \in \delta$ such that $\varphi ( \Om_{\mu}^2 ) (\mu) = 1$, we get a basis $(\xi_\ell)$ of $\Hilb_k$ satisfying Equations (\ref{eq:rep_M_L}). Let $\Hilb_k^{\alt}$ be the subspace of $\Hilb_k$ consisting of alternate sections, that is the sections $\Psi$ satisfying
$$ \Psi( -x) =  - \Psi (x) .$$
Since $\xi_{\ell } (-x) = \xi _{- \ell} (x)$,  the family $( \xi_\ell - \xi_{-\ell})_{\ell =1, \ldots , k-1}$ is a basis of $\Hilb_k^{\alt}$. Let $\Phi$ be the isomorphism from $V_k (\Si)$ to $\Hilb_k^{\alt}$  sending $Z_k ( N_K, L , \ell)$ into $\frac{1}{\sqrt 2} ( \xi_\ell - \xi_{-\ell})$. Since $J_{\ell} = - J_{\ell}$, the image of $Z_k (E_K)$ by $\Phi$ is $Z_k$, compare Equations (\ref{eq:def_knot_State_top}) and (\ref{eq:def_Z_k}). As a last remark, $\Phi$ is uniquely defined up to a power of $i$. This indeterminacy comes here from the choice of the orientation of the longitude and the meridian and of $\Om_\mu$. For additional properties of these morphisms, we sent the reader to Section 2.3 of \cite{LJ1}. 

The lattice $R= \la \Z \oplus \mu \Z $ and the group $\Z_2$ act on $E$ by translation and multiplication by $\pm 1$ respectively. This defines an action of the semi-direct product $R \rtimes \Z_2$ that can be lift to the prequantum bundle $L$ and the half-form bundle $\delta$ as follows. For any $(x,\ep) \in R  \times \Z_2$, we let
\begin{gather*} 
(x, \ep ) . (y, v) = (x + \ep y , \ep v) , \qquad (y, v) \in E \times \delta \\ 
(x, \ep ) . (z, w) = ( x + z , e ^{\frac{i}{2} \om ( x, z)} w)  , \qquad (z, w) \in L 
\end{gather*}
The quotient of $\delta$ and $L$ by these actions are line (orbi)-bundles, that we denote respectively by $\delta$ and $L_{\CS}$. The space $\Hilb_k^{\alt}$ is naturally isomorphic to the space of holomorphic sections of $L_{\CS}^k \otimes \delta$. As we will see in Section \ref{sec:peripheral-torus}, the quotient of $E$ by $R \rtimes \Z_2$ is naturally isomorphic to the moduli space $\mo (E_K)$, so that the knot state may be considered as a section over $\mo ( \Si)$.

\subsection{Recurrence relation}
 
It was observed by Hikami in \cite{hikami} that the $J_\ell$ satisfy some $q$-difference relations. These relations lead to the following characterizations of the knot state defined in Equation (\ref{eq:def_Z_k}). 

\begin{theo} \label{sec:knot_state_charac}
The state $Z_k$ of the torus knot with parameter $(a,b)$ satisfies the inhomogeneous equation 
\begin{gather}  \label{eq:equation_non_homogene}
  Z_k = ( q^{-1} M)^{-2 ab} L^{-2} Z_k + \sum_{i =1,\ldots, 4} \ep_i q^{\ep_i}  ( q^{-1} M )^{- ab + p_i } Z^0_k
\end{gather}
where $Z_0^k $ is a vector of $\Hilb_k$ such that 
\begin{gather} \label{eq:car_Z^0}
 L Z^0_k = Z^0_k , \qquad Z^0_k ( 0 ) = \frac{e^{ i \frac{3\pi}{4} }}{\sqrt 2} \Bigl( \frac{k}{2 \pi} \Bigr)^{1/4} \Om _\lambda + O(k^{-\infty}),
\end{gather}
$\Om_\la \in \delta$ is such that $\varphi( \Om_\la ^2) ( \la) =1$ and the $( \ep_i , p_i )$, $i=1,2,3,4$ are respectively given by $(1 , -a-b)$, $( -1, -a + b)$, $(1, a+ b)$ and $(-1, a-b)$. 
\end{theo}

 An important point is that this characterization is independent of the basis $(\xi_\ell)$. In the sequel we will actually work with an other basis. 

\begin{proof} 
With a straightforward computation one checks that the Jones polynomials satisfy the following recurrence relation:
\begin{xalignat}{1}  \notag
 J_\ell (t) &  = t^{ 4 ab ( 1 - \ell )} J_{\ell -2 } ( t) + \frac{ t^{ 2 ab ( 1 - \ell)}}{ t^2 - t ^{-2}} \bigl( t^{-2 ( a + b ) ( \ell -1 ) + 2} -  t^{ -2 ( a - b) ( \ell -1 ) -2 }  \\ &
+ t^{ 2 ( a + b ) ( \ell -1 ) + 2 } - t^{ 2 ( a -b ) ( \ell -1) -2 } \bigr) \notag 
\end{xalignat}
Consequently  $Z_k$ satisfies the inhomogeneous equation (\ref{eq:equation_non_homogene}) where $Z^0_k$ is the state of $\Hilb_k$ given by 
\begin{xalignat*}{2} \notag 
  Z^0_k = & \frac{\sin \bigl( \pi/k \bigr) }{\sqrt{k}}.\frac{1}{q - q^{-1} } \sum _{\ell \in \Z / 2k
  \Z} \xi_\ell \\ 
=  & \frac{ -i }{2 \sqrt{ k }} \sum_{\ell \in \Z / 2k
  \Z} \xi_\ell 
\end{xalignat*}
We deduce the estimation of $Z_k^0(0)$ from Equation (\ref{eq:phi_asymptot}).
\end{proof}

\section{Knot state analysis} \label{sec:knot-state-analysis}

\subsection{Change of lattice} 

Let $D= 2 ab$. To handle equation (\ref{eq:equation_non_homogene}), we introduce the lattice $$R_D = D \mu \Z \oplus \la \Z$$ and the vector space $\Hilb_{D,k}$ consisting of holomorphic sections of $L^k \otimes \delta$ such that $ T_{x}^* \Psi = \Psi$ for any $x \in R_D$. Since $R_D$ is contained in $ \mu \Z \oplus \la \Z$, the space $\Hilb_k$ we considered in Section \ref{sec:def_knot_state} is a subspace of $\Hilb_{D, k }$, actually
\begin{gather} \label{eq:carH}
 \Hilb_{k} = \Hilb_{D,k} \cap \ker (T_\mu^* - \id)
\end{gather} 
By section \ref{sec:finite-heis-group}, $\Hilb_{D,k}$  is $2kD$-dimensional. Introduce the following endomorphisms of $\Hilb_{D,k}$ 
\begin{gather}  \label{eq:def_R_S_q^1/D} 
S = T^*_{- \frac{1}{2kD} \la} , \quad R = T^*_{ \frac{ 1}{2kD} ( D \mu - 2 \lambda) } \quad \text{ and } \quad q ^{ \frac{1}{D}} = e^{ i \frac{ \pi }{kD}}  
\end{gather}
Since $(-\la, D \mu - 2 \la)$ is an oriented basis of $R_D$, $\Hilb_{D,k}$ admits a basis  $(\Psi_\ell ; \; \ell \in \Z / 2k D \Z)$ such that
\begin{gather} \label{eq:base_H_DDDD}
 S \Psi_{\ell} = q^{\frac{\ell}{D}  } \Psi_\ell, \qquad R \Psi_{\ell} = \Psi_{\ell +1}
\end{gather}
and 
\begin{gather} \label{eq:norm_phase}
\Psi_0( 0 ) = e^{i \frac{\pi}{4}} \Bigl( \frac{k}{2 \pi } \Bigr)^{ 1/4} \Om_\la + O( k^{- \infty} ),
\end{gather}
$\Om_\la$ being the vector given in Theorem \ref{sec:knot_state_charac}.

\begin{theo}  \label{theo:change-lattice}
For any positive integer $k$, the state $Z_k$ of the torus knot with parameter $(a,b)$ satisfies
$$  (\id - R^{-D} ) Z_k =  \al (k) Y_k   \mod E_k^\perp $$ 
where $E_k$ is the subspace of $\Hilb_k$ generated by the family $( \Psi_n, - 2 k - ab + a + b \leqslant n < 2k -ab -a -b )$, $Y_k$ is the vector of $\Hilb_{D,k}$ given by 
$$Y_k = R^{-ab} (R^a - R^{-a} ) ( R^b - R^{-b} ) \Psi_0 = \sum_i \ep_i \Psi_{-ab + p_i}$$
with $(\ep_i, p_i)$ defined as in Theorem \ref{sec:knot_state_charac}
and $(\al (k))$ a sequence of complex numbers satisfying $$ \al (k) = \frac{i }{\sqrt 2} q^{-\frac{a^2 + b^2}{D} + \frac{1}{4}D} + O(k^{-\infty}).$$ 
\end{theo}

For the proof we establish various lemmas. 

\begin{lem} \label{lem:M_R_S}
For any integer $r$, we have: $M^r = q^{\frac{r^2}{D}} S^{-2 r } R ^{r} = q^{ -\frac{r^2}{D}} R^r S^{-2r}$.
\end{lem}

\begin{proof} 
Formula (\ref{eq:loi_Heisenberg}) implies that
$$ T_x ^* T_y^* =  e^{\frac{ik}{2}  \om ( x, y) } T_{x+ y} ^* .$$ The results follows easily from the definition of $M$, $R$ and $S$, cf. Equations (\ref{eq:def_M_L_q}) and (\ref{eq:def_R_S_q^1/D}). 
\end{proof}

\begin{lem} \label{lem:Z0_nouvelle_base}
There exists sequences $\al_m(k), m \in \Z/ D \Z$ such that
$$ Z_k^0 = \sum_{m \in \Z / D \Z } \al_m (k)  \Psi_{2k m}  .$$
Furthermore $\al_0 (k) = \frac{i}{\sqrt2} + O(k^{-\infty})$.
\end{lem}
\begin{proof} 
By Theorem \ref{sec:knot_state_charac}, $Z_k^0$ belongs to the line $\Hilb_k \cap \ker ( T^*_{\la/2k } - \id )$. Let us check that this eigenline is generated by the vector
$$P_k =  \sum_{m \in \Z / D \Z}  T^*_{m \mu} \Psi_0 .$$
Since $T^*_{D \mu } \Psi_0 = \Psi_0$, we have that $T^*_{\mu} P_k = P_k$, so by (\ref{eq:carH}), $P_k$ belongs to $\Hilb_k$.  
Applying lemma \ref{lem:M_R_S} to $T^*_{\mu} = M^{2k}$, we obtain that $T^*_{m \mu} \Psi_0$ is a multiple of $\Psi_{2mk}$. By Equation (\ref{eq:base_H_DDDD}), the eigenspace $\Hilb_{D,k} \cap \ker ( T^*_{\la/2k} -\id ) $ is generated by the vectors $\Psi_{2mk}$, where $m$ runs overs $\Z$ mod $D\Z$. So $P_k$ belongs to $\Hilb_k \cap \ker ( T^*_{\la/2k } - \id )$. 

To compare $P_k$ and $Z_k^0$, we estimate their value at $0$. If $m \neq 0$, it follows from Equation (\ref{eq:estim_1_psi_0}) that $( T^*_{m \mu} \Psi_0)(0)$ is a $O(k^{-\infty})$. So by equation (\ref{eq:norm_phase}), 
$$ P_k(0) = e^{i \frac{\pi}{4}} \Bigl( \frac{k}{2 \pi } \Bigr)^{ 1/4} \Om_\la + O( k^{- \infty} ).$$ 
The value of $Z_k^0$ at $0$ is given in Equation (\ref{eq:car_Z^0}). So $Z_k^0 = \la_k P_k$ with $\la_k = \frac{i}{\sqrt 2} + O(k^{-\infty})$. This proves the result.       
\end{proof}

\begin{lem}  \label{lem:second_membre}
There exists sequences $\al_{m,i}(k), m \in \Z/D \Z, i =1, \ldots ,4 $ such that for any $i$, 
$$ q^{\ep_i}  ( q^{-1} M )^{- ab + p_i } Z^0_k =   \sum_{m \in \Z / D \Z}  \al_{m,i} (k) \Psi_{2km -ab + p_i} .$$
Furthermore,  $\al_{0,i} (k)$ does not depend on $i$ and is given by 
$$ \al_{0,i} (k) = \al_0 (k)  q^{-\frac{a^2 + b^2}{D} + \frac{1}{4}D} .$$ 
 \end{lem}
\begin{proof} 
By lemma \ref{lem:Z0_nouvelle_base},
$$  q^{\ep_i}  ( q^{-1} M )^{- ab + p_i } Z^0_k =  \sum_{m \in \Z / D \Z } \al_m (k)   q^{\ep_i}  ( q^{-1} M )^{- ab + p_i }  \Psi_{2k m}  .$$ 
By lemma \ref{lem:M_R_S}, $ ( q^{-1} M)^r \Psi_{2mk}$ is a multiple of $\Psi_{2mk + r }$, which shows the first part of the result. Let us compute this multiple for $m=0$, 
$$ ( q^{-1} M)^r \Psi_0 = q ^{-r + \frac{r^2}{D}} S^{-2r} \Psi_r = q ^{-r - \frac{r^2}{D}} \Psi_r$$
By a straightforward computation, we get 
\begin{gather*} 
\ep_i - r_i - \frac{ r_i^2}{D} = - \frac{ a^2 + b ^2}{D} +  \frac{1}{4} D
\end{gather*}
where $r_i = -ab + p_i$. So 
$$ q^{\ep_i}  ( q^{-1} M )^{- ab + p_i } \Psi_0 =   q^{-\frac{a^2 + b^2}{D} + \frac{1}{4}D}  \Psi_{-ab + p_i}$$
and the result follows. 
\end{proof}

\begin{proof}[Proof of Theorem \ref{theo:change-lattice}]
By lemma \ref{lem:M_R_S}, equation (\ref{eq:equation_non_homogene}) is equivalent to
\begin{gather*}
 ( \id - R^{-D} ) Z_k =  \sum_{i =1,\ldots, 4} \ep_i q^{\ep_i}  ( q^{-1} M )^{- ab + p_i } Z^0_k
\end{gather*}
By Lemma \ref{lem:second_membre}, it comes
\begin{xalignat*}{2} 
 ( \id - R^{-D} ) Z_k  = &   \sum_{i =1,\ldots, 4} \sum_{ m \in \Z/ D \Z } \ep_i  \al_{m,i}(k)   \Psi_{2mk -ab + p_i}  \\ =& 
 \sum_{i =1,\ldots, 4}  \ep_i  \al_{0,i}(k)   \Psi_{-ab + p_i}  \mod E^\perp_k
\end{xalignat*}  
Here we used that $|p_i| \leqslant a+ b$, so that the vectors $\Psi_{2mk -ab + p_i}$ belong to $E_k^{\perp}$ when $m \neq 0$.  
Let us remind that the coefficients $\al_{i,0} (k) $ do not depend on $i$. Set $\al(k) = \al_{0,i}(k)$.  We have $ \Psi_{-ab + p_i} = R^{-ab + p_i} \Psi_0$ and
$$ \sum \ep_i R^{ p_i} =  ( R^a - R^{-a}) ( R^b - R^{-b})$$
and the result follows. 
\end{proof}

\subsection{Solution of the recurrence relations} 

From now on, we denote by $(\Psi_\ell, \; \ell \in \Z / 2kD \Z )$ the basis of $\Hilb_{D,k}$ determined by~(\ref{eq:base_H_DDDD}) and  (\ref{eq:norm_phase}). 
Introduce the vectors of $\Hilb_{D,k}$
\begin{gather}  \label{eq:def_phil}
\Phi_\ell = \frac{1}{\sqrt { 2 k D}} \sum _{n \in \Z / 2k D \Z} e^{ \frac{2i \pi }{D} n \ell} \Psi_n, \qquad \ell  \in \Z / D \Z
\end{gather}
Our aim is to prove the following theorem
\begin{theo} \label{theo:quantum} 
There exists sequences $\ga_\ell (k)$, $\ell \in \Z / D \Z$ such that for any positive integer $k$, the state $Z_k$ of the torus knot with parameter $(a,b)$ satisfies  
$$ Z_k = \sum_{\ell \in \Z/ D \Z} \ga_\ell(k) \Phi_\ell \mod E_{k,+}^\perp$$ 
where $E_{k, +}$ is the subspace of $\Hilb_k$ generated by the family $( \Psi_n, -ab + a + b \leqslant n < 2k -ab -a -b )$.
Furthermore the $\ga_\ell(k)$'s verify the following equations
\begin{gather*}
   \ga_{-1} (k) = 0 , \qquad \ga_0 (k) = \frac{1}{2} \be_0(k) \qquad \text{and} \\
\forall \ell, \quad  \ga_{\ell } (k) = e^{ \frac{2i \pi}{D} 2k ( \ell - 1 )} \ga_{\ell-2}(k)  + \be_\ell (k) ,
\end{gather*}
where   for any $\ell \in \Z / D \Z$, 
$$\be_\ell (k)  =  C_k (-1)^\ell \sin \Bigl( \frac{\pi \ell }{a} \Bigr) \sin \Bigl( \frac{\pi \ell }{b} \Bigr) + O(k^{-\infty}) $$
with $C_k = - 4 \al (k) \bigl( \frac{2k }{D} \bigr)^{1/2}  $, the constant $\al (k)$ being given in Theorem \ref{theo:change-lattice}. 
\end{theo}
Observe that the $\Phi_\ell$'s can be characterized up to a $O(k^{-\infty})$ by the following equation 
\begin{gather} \label{eq:car_phi_0}
 R \Phi_0 = \Phi_0 , \qquad \Phi_0 ( 0 ) = i \Bigl( \frac{k}{2 \pi } \Bigr)^{1/4} \Om_{D \mu - 2 \la} + O(k^{-\infty}) \\ 
\label{eq:car_phi_ell}
 \Phi_\ell = S ^{2k \ell } \Phi_0 , \qquad \ell \in \Z / D \Z
\end{gather}
Here the estimate of $\Phi_0$ follows from equation (\ref{eq:phi_asymptot}) taking into account the difference in the normalization of $\Psi_0$ (we used $\Om_{\la}$ instead of $\Om_{- \la}$).

The remainder of this section is devoted to the proof of Theorem \ref{theo:quantum}.
Let $E_{k,+}$ be the subspace of $\Hilb_k$ introduced in Theorem \ref{theo:quantum} and let $E_{k, -}$ be the subspace generated by the family $( \Psi_n, -2k - ab + a  + b \leqslant n < -ab -a -b )$.  

\begin{lem} \label{lem:coefs}
For any positive $k$, there exists sequences $\ga^+_\ell , \ga^{-}_\ell , \; \ell \in \Z / D \Z$ such that  
$$  Z_k =  \sum_{\ell \in \Z/ D \Z} \ga_\ell^+ \Phi_\ell   \mod E_{k, +}^\perp$$
and 
$$ Z_k =  \sum_{\ell \in \Z/ D \Z} \ga_\ell^- \Phi_\ell \mod E_{k, -}^{\perp}.$$ 
\end{lem}

\begin{proof} 
Recall that for any $i$, $|p_i| \leqslant a +b$. So by Theorem \ref{theo:change-lattice}, 
$$( \id - R^{-D} ) Z_k  = 0 \mod E_{k, \pm} ^\perp.$$ 
Thus the coefficients $\langle Z_k  , \Psi_n \rangle $ of $Z_k$  coincide with a $D$-periodic sequence when $\Psi_n \in E_{k , \pm}$.    
Since any $D$-periodic sequence is a linear combination of the sequences $\exp (2i \pi \frac{ n\ell}{D})$, $\ell \in \Z / D \Z$, the result follows. 
\end{proof}  

\begin{lem} \label{sec:saut}
We have $\ga_\ell^+ = \ga_\ell ^- + \be_\ell$, where $\be_\ell$ satisfies 
$$ \be_\ell = 4  \al (k)  \sqrt{ \frac{2k }{D}}   (-1)^{\ell+1} \sin \Bigl( \frac{\pi \ell }{a} \Bigr) \sin \Bigl( \frac{\pi \ell }{b} \Bigr)  .$$ 
\end{lem}
\begin{proof} 
This is again a consequence of Theorem \ref{theo:change-lattice}. Denote by $P$ the endomorphism of $\Hilb_{D,k}$ defined on the basis $(\Psi_{n})$ by
$$ P ( \Psi_n) =  \sum_{ m \in \Z / 2k \Z} \Psi_{n + mD}$$
The coefficient $\be_\ell$ are such that 
\begin{gather} \label{eq:1} 
  \sum_{\ell \in \Z / D \Z} \be_\ell \Phi_\ell = \al (k) P ( Y_k) 
\end{gather}
We can extract the $\be_\ell$'s from this equation by a discrete Fourier transform. We have 
$$ P(\Psi_r) =  \sqrt{ \frac{2k }{D}} \sum_{\ell \in \Z / D \Z} e^{-  \frac{2 i \pi }{ D} \ell r } \Phi_\ell.$$ 
Consequently 
\begin{xalignat*}{2}
 \be_\ell = & \sqrt{ \frac{2k }{D}} \al (k) \sum_{i = 1} ^4  \ep_i  e^{  -  \frac{2 i \pi }{ D} \ell (-ab + p_i) }\\
= &\sqrt{ \frac{2k }{D}} \al (k) (-1)^{\ell}  \sum_{i = 1} ^4  \ep_i  e^{  -  \frac{2 i \pi }{ D} \ell  p_i  }
\end{xalignat*}
Furthermore we have the equality
$$ \sum_{i=1}^4 \ep_i e^{-  \frac{2 i \pi }{ D} \ell p_i } = - 4  \sin \Bigl( \frac{\pi \ell }{a} \Bigr) \sin \Bigl( \frac{\pi \ell }{b} \Bigr)
$$
which leads to the conclusion. 
\end{proof}

\begin{lem}\label{sec:symetrie}
For any $\ell$, we have $\ga_{- \ell}^- = - \ga_{\ell}^+$.
\end{lem}

\begin{proof}
Since $Z_k ( -x) = - Z_k (x)$ and $\Psi_n ( -x) = \Psi_{-n}(x)$, we have
$$ \bigl\langle Z_k , \Psi_{-n} \bigr\rangle =  -  \bigl\langle Z_k , \Psi_{n} \bigr\rangle $$ 
Furthermore, $ \langle \Phi_\ell, \Psi_{-n} \rangle = \langle \Phi_{- \ell} , \Psi_n \rangle $. So 
$$\bigl\langle Z_k - \sum_{\ell \in \Z/ D \Z} \ga_\ell^+ \Phi_\ell , \Psi_n \bigr\rangle = 0$$
implies that 
$$\bigl\langle Z_k + \sum_{\ell \in \Z/ D \Z} \ga_\ell^+ \Phi_{-\ell} , \Psi_{-n} \bigr\rangle = 0$$
and the result follows from Lemma \ref{lem:coefs}.
\end{proof}

\begin{lem} \label{sec:periodocite}
For any $\ell$, we have  $ \ga^-_{\ell } = e^{ \frac{2i \pi }{D} 2k ( \ell -1 ) } \ga_{\ell - 2} ^+$. \end{lem}

\begin{proof} 
Since $Z_k$ belongs to $\Hilb_k$, it satisfies $M^{2k} Z_k = Z_k$. By lemma \ref{lem:M_R_S}, $$M^{2k} = e^{\frac{2i \pi}{D} 2k} S^{ -4k} R^{2k}.$$ Consequently $M^{2k} \Psi_n$ is a multiple of $\Psi_{n + 2k }$ and 
$$ M^{2k} \Phi_\ell =  e^{ \frac{2i \pi}{D} 2k ( 1 -\ell ) } S^{-4k} \Phi_{\ell} =   e^{ \frac{2i \pi}{D} 2k ( 1 -\ell ) } \Phi_{\ell -2}.$$
Since $M^{2k}$ is an isometry, 
$$\bigl\langle Z_k - \sum_{\ell \in \Z/ D \Z} \ga_\ell^- \Phi_\ell , \Psi_n \bigr\rangle = 0$$
implies that 
$$ \bigl\langle Z_k - \sum_{\ell \in \Z/ D \Z} \ga_\ell^-  e^{ \frac{2i \pi}{D} 2k ( 1 -\ell ) } \Phi_{\ell -2}  , \Psi_{n+2k}  \bigr\rangle = 0.$$
By Lemma \ref{lem:coefs}, we obtain $\ga^+_{\ell -2 } = \ga_{\ell} ^- e^{ \frac{2i \pi}{D} 2k ( 1 -\ell ) }$. 
\end{proof}

\begin{lem} 
We have $\ga_0^+ = \frac{1}{2} \be_0$, $\ga_{-1}^+ = 0$ and for any $\ell$
$$  \ga_{\ell }^+ = e^{ \frac{2i \pi}{D} 2k ( \ell - 1 )} \ga_{\ell-2}^+  + \be_l. $$
\end{lem}
\begin{proof} 
Lemma \ref{sec:symetrie} for $\ell = -1$ and Lemma \ref{sec:periodocite} for $\ell = 1$ imply that $\ga_{-1}^+ = 0 .$
By lemmas \ref{sec:saut}, $\ga_0^+ = \ga_0^- + \be_0$ and by lemma \ref{sec:symetrie}, $\ga_0^{-} = - \ga_{0}^+ $. So $$\ga_0^+ = \frac{1}{2}\be_0  $$
Finally lemmas \ref{sec:saut} and \ref{sec:periodocite} imply the recurrence relation. 
\end{proof}

\subsection{Microlocal properties} 

We will now deduce the asymptotic behavior of the knot state. We will consider separately the following subsets of $E$ 
$$ A = ( -1 , 0  ) \mu + \R \la , \qquad  B = ( \R \setminus D^{-1} \Z) \la , \qquad C= D^{-1} ( a \Z \cup b \Z ) \la.
 $$

\subsubsection{On the set $A$} 
Let us start with the asymptotic behavior of the $\Phi_\ell$'s defined in Equation (\ref{eq:def_phil}). 

\begin{prop} \label{prop:les_phi_ell}
For any $\ell \in \Z/D\Z$, the microsupport of $\Phi_\ell$ is $\frac{\ell}{D} \la+ \lambda \Z + (D\mu - 2 \lambda) \R$. Furthermore for any $\delta \in (0,1)$, there exists a positive $C$ such that for all $x \in \frac{\ell}{D} \la + ( -\frac{1}{2}, \frac{1}{2})\lambda + (D\mu - 2 \lambda) \R$, we have
$$ \Bigl|  \Phi_\ell (x) -  i \Bigl( \frac{k}{2\pi}\Bigr)^{1/4}   T^*_{ - \ell \lambda /D } t_{D \mu - 2 \la}^k(x)  \otimes \Om_{D\mu - 2 \la} \Bigr| \leqslant C e ^{-k/C}
$$ 
where $t_{D \mu - 2 \la}$ is the holomorphic section of $L$ equal to 1 on the line $( D \mu  - 2 \la ) \R$ and  $\Om_{D \mu - 2 \la } \in \delta$ is such that $\varphi( \Om^2 _{D \mu - 2 \la} )( D \mu - 2 \la) = 1$.
\end{prop}

\begin{proof}
For $\ell =0$, this is a consequence of (\ref{eq:estim_1_psi_0}) and (\ref{eq:estim_2_psi_0}) applied to the basis $( D\mu - 2 \lambda, \lambda)$  of $R_D$. Indeed $\Phi_0$ is an eigenstate of $T^*_{(D \mu - 2 \la )/ 2kD}$ with eigenvalue 1. Its value at $0$ is given in (\ref{eq:car_phi_0}). 
Then using that $\Phi_\ell = T^*_{-\ell \la /D} \Phi_0$, the result follows for any $\ell$. 
\end{proof}

Let $A$ be the open set $( -1 , 0  ) \mu + \R \la$ of $E$. Introduce the open intervals 
\begin{gather} \label{eq:intervalle_ell}
 I_\ell = \Bigl\{ \ell \frac{\la}{D} +  t  \Bigl( \frac{\la}{D} - \frac{\mu}{2}  \Bigr)   \Bigl/ \; t \in (0,2) \Bigr\}, \quad \ell \in \Z  
\end{gather}
Then we deduce from Proposition \ref{prop:les_phi_ell}, Proposition \ref{prop:estim} and Theorem  \ref{theo:quantum} the following

\begin{theo}  \label{theo:irreductible}
The microsupport of the state $(Z_k)$ of the torus knot with parameter $(a,b)$  does not intersect $A \setminus \cup_{\ell \in \Z} I_\ell$. Furthermore for any $\ell \in \Z$, for any $x_0 \in I_\ell$, there exists a neighborhood $V$ of $x_0$, such that 
$$ Z_k(x) = \Bigl( \frac{k}{2\pi}\Bigr)^{1/4} i  \ga_\ell (k)  T^*_{ - \ell \lambda /D } t_{D \mu - 2 \la}^k(x)  \otimes \Om_{D\mu - 2 \la}  + O(k^{-\infty}) 
$$ 
on $V$  where $t_{D \mu - 2 \la}$ is the holomorphic section of $L$ equal to 1 on the line $( D \mu  - 2 \la ) \R$, $\Om_{D \mu - 2 \la } \in \delta$ is such that $\varphi( \Om^2 _{D \mu - 2 \la} )( D \mu - 2 \la) = 1$ and the $O(k^{-\infty})$ is uniform on $V$. 
\end{theo}

\subsubsection{On the set $B$} \label{sec:abel-repr}

The Alexander polynomial of the torus knot with parameter $(a,b)$ is 
$$ \Delta_{a,b}  (t ) = t^{\frac{1}{2}( a + b -ab -1 )} \frac{(t-1)(t^{ab} -1 ) }{ ( t^a -1 ) ( t^b -1 )} $$
Introduce the holomorphic section $t_\la$ of $L$ which restricts to the constant section equal to 1 on $\R \la$. Let $\Om_{\la} \in \delta$ chosen as in Theorem \ref{sec:knot_state_charac}. It is characterized up to sign by the condition $\Om_{\la}^2 ( \la ) =1 $. The following theorem has been proved in \cite{LJ1}. 

\begin{theo}  \label{theo:abelian}
For any $x_o \in \la ( \R \setminus \frac{1}{D} \Z)$, there exists a neighborhood $U \subset E$ of $x_o$ and a sequence $(f( \cdot , k))$ of $\Ci (U, \C)$ such that the state $Z_k$ of the torus knot with parameter $(a,b)$ 
$$ Z_k (x)   =    \Bigl( \frac{k}{2 \pi} \Bigr)^{1/4}  t_\la^k(x) \otimes f(x ,k ) \Om_{\la} + O(k^{-\infty}) , \qquad \forall x \in V $$
where the $O$ is uniform with respect to $x$. Furthermore $f(\cdot, k)$ has an asymptotic expansion for the  $\Ci$ topology  of the form $f_0 + k^{-1} f_1 + \ldots $  with coefficients $f_i \in \Ci (U)$.  The leading term satisfies 
$$  f_0(q\la) =  \frac{e^{-i \frac{\pi}{4}}}{\sqrt{2}}  \frac{ \si - \si^{-1}  }{\Delta_{a,b} (\si^2)}  \quad \text{ with } \quad \si = e ^{ 2i \pi q }  $$ for any $q \la \in U$.
\end{theo}

We recall briefly the proof for further use. We will use tools of microlocal analysis which were developed in the articles \cite{oim03b} and \cite{oim03}. The properties we need are summarized in \cite{LJ1}. It follows from Theorem \ref{theo:change-lattice} and Proposition \ref{prop:estim}, that the knot state $Z_k$ is a microlocal solution of 
\begin{gather} \label{eq:non-homogene}
 (\id - R^{-D} ) Z_k = \al (k) R^{-ab} (R^a - R^{-a} ) ( R^b - R^{-b} ) \Psi_0 
\end{gather}
on the open set $ \R \la  + (-1,1) \mu$, cf. Section 5.1.2 of \cite{LJ2} for the notion of microlocal solution. 

Denote by $p$ and $q$ the linear coordinates of $E$ associated to the basis $(D\mu -2 \la, \la )$ and let $\si = \exp ( 2 i \pi q) $. 
By Theorem 3.1 of \cite{LJ1}, $\id - R^{-D}$ is a Toeplitz operators of $\Hilb_{D,k}$ with principal symbol $ 1 - \sigma ^D$. By assumption, $1 - \sigma^D$ does not vanish at $x_o$. So we can invert $\id - R^{-D}$ on a neighborhood of $x_0$ and deduce that
$$   Z_k = \al (k) T_k \Psi_0 
$$ 
on a neighborhood of $x_o$,  where $T_k$ is a Toeplitz operator with symbol $$( 1 - \si^{D})^{-1}  \si ^{ab} ( \si ^a - \si ^{-a}  )( \si ^b - \si  ^{-b})= - \frac{ \si - \si ^{-1}} { \Delta_{a,b} ( \si^2) }.$$ 
To conclude we use the fact that $\Psi_0$ is a Lagrangian state, cf. Section \ref{sec:basis-lagrangian}, and that we know how Toeplitz operators acts on Lagrangian states, Proposition 2.7 of \cite{oim03}. The computation of $f_0$ follows from the normalization (\ref{eq:norm_phase}) and the fact that $\al( k) = \frac{i }{\sqrt{2}} + O(k^{-1})$.

\subsubsection{On the set $C$}

\begin{theo}\label{theo:abelien_irreductible}
Let $\ell \in ( a \Z \cup b \Z)$ and $x_o =  \frac{\ell}{2 ab } \la $. Then there exists a neighborhood  $U \subset E$ of $x_o$ such that for all $x \in U$,  
\begin{xalignat*}{2} 
 Z_k & (x)   =     \Bigl( \frac{k}{2\pi}\Bigr)^{1/4} i  \ga_\ell (k)  T^*_{ - \ell \lambda /D } t_{D \mu - 2 \la}^k(x)  \otimes \Om_{D\mu - 2 \la}
\\
+ & \Bigl( \frac{k}{2 \pi} \Bigr)^{1/4}  t_\la^k(x) \otimes f(x ,k ) \Om_{\la} + O(k^{-\infty}) , \qquad \forall x \in V 
\end{xalignat*} 
where the $O$ is uniform with respect to $x$, $\ga_k ( \ell)$ is defined in Theorem \ref{theo:quantum}, $t_{D \mu - 2 \la}$ is a section of $L$ and $\Om_{D\mu - 2 \la}$ a vector in $\delta$ satisfying the same assumptions as in Theorem \ref{theo:irreductible}, and $f(\cdot, k)$ is a sequence of $\Ci (U,C)$ satisfying the same assumptions as in theorem  \ref{theo:abelian}. 
\end{theo}

Any solution of Equation (\ref{eq:non-homogene}) is the sum of a particular solution and a solution of the homogeneous equation
 \begin{gather} \label{eq:homogene}
 (\id - R^{-D} ) \Psi_k = 0 
\end{gather}
The following lemma describes the solution of this latter equation on the open subset 
$$U = x_o + \tfrac{1}{2D} (-1,1) \la + \tfrac{1}{2}  (-1,1) ( \mu - 2/D \la) $$
of $E$. This does not require that $\ell \in ( a \Z \cup b \Z)$, this assumption will be  used for the construction of a particular solution. 

\begin{lem} \label{lem:solution_homogene}
The microlocal solution of Equation (\ref{eq:homogene}) on $U$ are of the form
$$ \Psi_k =   \la_k  T^*_{ - \ell \lambda /D } t_{D \mu - 2 \la}^k(x)  \otimes \Om_{D\mu - 2 \la}+ O(k^{-\infty}) $$
where the $O$ is uniform on compact subsets of $U$, the integer $\ell$, the section $t_{D \mu - 2 \la}$ and $\Om_{D\mu 2 -2 \la} \in \delta$ are defined as in Theorem \ref{theo:irreductible} and $\la_k$ is a sequence of complex numbers which is $O(k^{m})$ for some $m$.
\end{lem}

\begin{proof} 
By Proposition \ref{prop:les_phi_ell},  $T^*_{ - \ell \lambda /D } t^k(x)  \otimes \Om_{D\mu - 2 \la}$ is a microlocal solution of Equation (\ref{eq:homogene}) on $U$. Furthermore this section is not $O(k^{-\infty})$ over   
$$I : = \Bigl\{ x_o + t( \mu -\frac{D}{2} \la ); \; t \in (- \tfrac{1}{2},  \tfrac{1}{2}) \Bigr\} .$$ 
Let ${\mathcal{M}}$ denote the space of microlocal solutions of (\ref{eq:homogene}) and ${\mathcal{M}} \cap O(k^{-\infty})$ the subspace of solutions which are $O(k^{-\infty})$ uniformly on any compact subset of $U$. ${\mathcal{M}}$ is a module over the ring  ${\mathcal{R}}$  of complex valued sequences which are $O(k^{m})$ for some $m$. To conclude it suffices to show that ${\mathcal{M}} / {\mathcal{M}} \cap O(k^{-\infty})$ is free and has dimension one over ${\mathcal{R}}$. 

Introduce as in Section \ref{sec:abel-repr} the function $\si$ and recall that $\id - R^{-D}$ is a Toeplitz operators of $\Hilb_{D,k}$ with principal symbol $ 1 - \si^ D$. Observe that the zero level set of $1- \si ^D$ intersects $U$ in the connected set $I$. If the symbol $\si$ took real values, we could directly conclude that ${\mathcal{M}}/O(k^{-\infty})$ is a one-dimensional module over ${\mathcal{R}}$, cf. Theorem 5.2 of \cite{LJ2}.  Actually this conclusion holds true. It can be proved by writing $I - R^{-D}$ over $U$ on the form $f_k(T)$ with $T$ a Toeplitz operator with principal symbol $q$ and $(f_k) $ a sequence in $\Ci( \R, \C)$ such that $f_k( x ) = 1 - \exp ( 2i \pi x D ) + O(k^{-1})$.
\end{proof}

To conclude the proof of Theorem \ref{theo:abelien_irreductible}, we construct a convenient particular solution. 

\begin{lem}\label{lem:particular_solution}
There exist a sequence $f(\cdot, k)$ satisfying the same assumptions as in Theorem~\ref{theo:abelien_irreductible} such that 
$$\Bigl( \frac{k}{2 \pi} \Bigr)^{1/4}  t_\la^k(x) \otimes f(x ,k ) \Om_{\la} + O(k^{-\infty}) $$
is a microlocal solution of Equation (\ref{eq:non-homogene}) over $U$.
\end{lem}

\begin{proof} 
Since $\ell \in (a\Z \cup b \Z)$, there exists two Laurent polynomials $P_1 $ and $P_2$ such that
$$ ( 1 - R^{-D} ) =   (e^{2i\pi \ell /D} - R) P_{1} (R) $$
 and 
$$ R^{-ab} ( R^a - R^{-a})( R^b - R^{-b} ) =  (e^{2i\pi \ell /D} - R) P_{2} (R) .$$
Any microlocal solution of 
\begin{gather} \label{eq:facteur}
 P_1 (R) \Psi_k = \al (k) P_2 (R) \Psi_0
\end{gather}
is a microlocal solution of (\ref{eq:non-homogene}). The remainder of the proof is the same as for Theorem \ref{theo:abelian}. 
\end{proof}

\section{Topological invariants} \label{sec:topol-invar}

\subsection{Character varieties} 

This chapter is based on the papers \cite{klassen} of Klassen and \cite{dk} of Dubois-Kashaev. 
Consider a torus knot $K$ with parameter $(a,b)$ and let $\mo ( E_K)$ be the space of representations of the knot exterior $E_K$ in $\su $ up to conjugation. 
The fundamental group $\pi$ of $E_K$ has the presentation $\{ x , y | x^a = y^b \}$. The class of the meridian is $\mu = x^m y ^n$ where $m$ and $n$ are integers satisfying $an + b m =1$.  

The subspace $\mo ^{\ab} (E_K)$ of $\mo ( E_K)$  consisting of abelian representations is homeomorphic to a segment $[0,1] \ni t$, the representation $\rho$ parametrized by $t$ is the abelian representation such that $\tr (\rho ( \mu )) = 2 \cos ( \pi t)$.  The subspace $\mo ^{\ir} (E_K )$  of $\mo (E_K)$ consisting of irreducible representations is the disjoint union of $(a-1)(b-1)/2$ open interval $I_{\al, \be}$. We will index these intervals  by the set of couples $(\al, \be) \in \{ 1, \ldots , a -1 \} \times \{ 1, \ldots , b-1 \}$ such that $\al$ and $\be $ have the same parity. The closure $\cl (I_{\al, \be})$ of $I_{\al, \be}$ consists of the irreducible representations $\rho$ such that $$\tr ( \rho (x)) = 2 \cos ( \al \pi /a ) , \qquad  \tr ( \rho (y)) = 2 \cos ( \be \pi /b ).$$ 
One deduces from the proof in \cite{klassen} the following parametrization: $ \cl (I_{\al, \be})$ is homeomorphic to the interval $[0,1] \ni t$, the representation $\rho$ corresponding to $t$ being given by
$$ \rho (x) = \exp (  \tfrac{\al }{2a} D ) , \quad \rho (x) =  R_{t} \exp (  \tfrac{\be}{2b}D )  R_{-t} $$ 
where $D $ is the diagonal matrix with entries $(2i \pi, - 2 i \pi)$ and $R_t$ is the rotation matrix of angle $t \pi /2$,
$$ R_t = \left( \begin{array}{cc}  \cos ( t \pi/2)  & -\sin ( t\pi/2)   \\  \sin ( t \pi/2) & \cos( t \pi /2) 
  \end{array} \right) .$$
$\rho$ is abelian if and only if $t = 0$ or $1$. 

By a straightforward computation, we have that
$$ \tr ( \rho (\mu)) =  2 \bigl( c^2 \cos ( \tfrac{\pi}{ab} ( \al bm + \be an ) ) +  s^2  \cos ( \tfrac{\pi}{ab} ( \al bm - \be an ) )  \bigr) $$
with $c = \cos  ( t \tfrac{\pi}{2} ) $ and $ s = \sin  ( t \tfrac{\pi}{2} )$. So we can alternatively parametrize $I_{\al, \be}$ by the conjugacy classes of $\rho ( \mu)$ as follows. Consider the quotient of $\Z $ by $(2ab \Z) \rtimes \Z_2$ where $(2ab \Z)$ acts by translation and $-1 \in \Z_2$ by $-\id _{\Z}$. Each equivalence class in this quotient has a unique representative in $\{0,\ldots , ab \}$.   Denote by $k^-_{\al, \be}$ and $k^{ +}_{\al, \be}$ the representatives of $[ \al bm + \be an]$ and $ [\al bm - \be an]$  in $\{0,\ldots , ab \}$ ordered in such a way that $k^-_{\al, \be} \leqslant k^{ +}_{ \al, \be}$. We will see in Lemma \ref{lem:k} that $$ 1 \leqslant k^- _{\al, \be} < k^{ +}_{\al, \be} \leqslant ab -1.$$  
We denote by $\rho_{\al, \be, t}$ the (class of a) representation in the closure of $ I_{\al, \be}$ such that 
$$ \tr ( \rho_{\al, \be, t} ( \mu )) = 2 \cos ( \tfrac{\pi}{ab} t ), \qquad   k^-_{\al, \be} \leqslant t  \leqslant   k^+_{\al, \be}  .$$
Let us study now the integers $k^+_{\al, \be}$ and  $k^-_{\al, \be}$

\begin{lem} \label{lem:morphism}
The group morphism 
$$\varphi: \Z / 2ab \Z \rightarrow (\Z/2a \Z) \times ( \Z / 2 b \Z), \qquad [k] \rightarrow ([k], [k])$$ is injective. Its image consists of the couples $([\al] , [\be])$ such that $\al$ and $\be$ have the same parity. If $\varphi ( [k]) = ([\al], [ \be])$ then $k = \al b m + \be a n $ mod $2ab \Z$.  
\end{lem}

Let ${\mathcal{P}}$ be the set of couples $(\al, \be) \in \{ 1, \ldots , a -1 \} \times \{ 1, \ldots , b-1 \}$ such that $\al$ and $\be $ have the same parity. 

\begin{lem} \label{lem:k}
The map from ${\mathcal{P}} \times \{ +, -\}$ to $\{ 0,1, \ldots, ab\}$ which sends $(\al, \be, \ep)$ into $k_{\al, \be}^\ep$ is injective. Its image is $\{ 0,1, \ldots, ab\} \setminus (a \Z \cup b \Z)$. 
\end{lem}

\begin{proof} 
The morphism $\varphi$ of Lemma \ref{lem:morphism} satisfies $\varphi(X) = Y$ with 
$$  Y = \Z / 2 ab Z \setminus ( (a\Z /2ab \Z) \cup (b\Z/2ab \Z) ).$$
and $Y$ the subset of  $\Z / 2aZ \times \Z / 2b \Z$ consisting of couples $(\al, \be)$ such that $\al$ and $\be$ have the same parity, $\al \neq 0,a $ and $\be \neq 0,b$.
The map $\Psi$ from ${\mathcal{P}} \times \{ \pm 1\} ^2$ to $X$ sending $(\al, \be, \ep_1, \ep_2)$ to $(\ep_1 \al, \ep_2 \be)$ is a bijection. Furthermore
$$ \Psi \circ \varphi^{-1}  \bigl( \{ ( \al , \be) \} \times \{ \pm 1 \}^2 \bigr) = \{ k^+_{\al, \be}, - k^+_{\al, \be}, k^-_{\al, \be}, -k^-_{\al, \be} \} .$$
From this, it is easy to conclude. 
\end{proof}

So the closures of the $I_{\al, \be}$'s are pairwise disjoint. The abelian representations which are limit of irreducible ones satisfy $\tr ( \rho ( \mu ) ) = 2 \cos ( \frac{\pi}{ab} \ell )$ where $\ell$ is an integer which is not a multiple of $a$ or of $b$. These representation are in one to one correspondence with the pair of conjugate roots of the Alexander polynomial.  

\subsection{Peripheral torus} \label{sec:peripheral-torus}

Let us consider now the space $\mo ( \Si)$  of representations of the fundamental group of the peripheral torus $\Si$ in $\su$. $\mo (\Si)$ is a quotient of $E = H_{1} (\Si, \R)$ in the following way. Let $\pi$  be the map from $E$ to $\mo (\Si)$  
$$ \pi (x) (\ga) = \exp ( ( \ga .x ) D), \qquad \ga \in H_{1} (\Si, \R) $$ 
where $D$ is the same matrix as above and the dot $.$ stands for the intersection product. $\pi$ induces a bijection between $\mo (\Si)$ and the quotient of $E$ by $R \rtimes \Z_2$, where $R = H_1 (\Si, \Z)$ acts by translation on $E$ and $-1 \in \Z_2$ acts by $-\id_E$. 

Let $\la$ be the longitude whose linking number with the knot vanishes. 
The curves $\la$ and $\mu$ may be viewed as classes in $H_1 (\Si, \Z)$, defining a basis. 
The restriction map $r$ from $\mo ^{\ab} (E_K)$ to $\mo ( \Si)$ is an embedding, its image being $\pi ( [0, \frac{1}{2} ] \la)$. 
It is known that $\la = x ^a \mu ^{-ab}$. From this one deduces that the restriction map $r$ from $\mo ^{\ir}  ( E_K) $ into $\mo (\Si)$ is given by 
$$ r ( \rho_{\al, \be, t} ) = \pi \Biggl( k^-_{\al , \be} \frac{\la}{D} + ( t - k ^-_{\al, \be}) \Bigl( \frac{\la}{D} - \frac{\mu}{2} \Bigr)   \Biggr) . $$ 
For any integer $\ell$ which has the same parity of $k^{\pm}_{\al, \be}$ and such that $k^{-} _{\al, \be} \leqslant \ell < k^+ _{\al, \be} $, let 
\begin{gather} \label{eq:def_Iell_al_be}
I^\ell_{\al, \be} = \bigl\{ \rho_{t, \al, \be} ; \; t\in (\ell, \ell +2 ) \bigr\} \subset I_{\al, \be} 
\end{gather}
Note that $r ( I_{\al, \be}^\ell) = \pi ( I_\ell)$ where $I_\ell$ is the interval defined in (\ref{eq:intervalle_ell}). 
Conversely, $r^{-1}( \pi ( I_{\ell}))$ is the union of the $I_{\ell , \al, \be }$ where $(\al, \be)$ runs over the set ${\mathcal{P}}_{\ell}$ given by  
\begin{gather} \label{eq:def_pl}
 {\mathcal{P}}_{\ell} = \bigl\{ ( \al , \be ) \in {\mathcal{P}}/ \; k_{\al, \be}^- \leqslant \ell < k_{\al, \be}^+  \text{ and } \al = \ell \mod 2 \bigr\}. 
\end{gather}

\subsection{Chern-Simons invariant}

For any representation $\rho \in \mo (E_K)$, the Chern-Simons invariant $\CS ( \rho)$  is defined as an element of the fiber at $r( \rho) $ of a line (orbi)-bundle $L_{\CS} \rightarrow \mo (\Si)$. As was explained in \cite{LJ2}, the bundle $L_{\CS}$ is naturally isomorphic to the quotient of the prequantum bundle $L\rightarrow E $ by the action of $R \rtimes \Z_2$. As a fact, the section of $r^* L_{\CS} \rightarrow \mo ( E_K)$ sending $\rho $ into $\CS ( \rho)$ is flat. Furthermore the Chern-Simons invariant of the trivial representation $\rho_0$ is the class of $1 \in L_{0} = \{ 0 \} \times \C$. Since $\mo (E_K)$ is connected, the Chern-Simons invariant is determined by these two properties. 

Recall that the asymptotic expansion of the knot state is given in terms of the section $t_{D\mu - 2 \la} $ of $L \rightarrow E$, which restriction to the line $(D \mu - 2 \la ) \R$ is constant equal to 1. Let us compute the Chern-Simons invariant in terms of this section.

\begin{lem}  \label{lem:Chern_simons}
For any integer $\ell$ with the same parity of $k^{\pm}_{\al, \be}$ and such that $k^{-} _{\al, \be} \leqslant \ell < k^+ _{\al, \be} $ and for any $t \in ( \ell, \ell +2)$, the Chern-Simons invariant at $\rho_{t, \al , \be}$ is the class of  
\begin{gather} \label{eq:Chern_Simons}
  e^{ \frac{i \pi}{D}( \ell - k^-_{\al, \be} ) ( \ell + k^-_{\al, \be}) } \bigl( T^*_{-\ell \la /D} t_{D \mu - 2\la} \bigr) \bigl( \ell  \tfrac{\la}{D}  + ( t - \ell ) \bigl( \tfrac{\la}{D} - \tfrac{\mu}{2} \bigr) \bigr) .
\end{gather}
\end{lem}

\begin{proof} 
For any line $V$ of $E$ which goes trough the origin, the constant sections of $L \rightarrow V$ are flat. Since the Heisenberg group acts by isomorphisms of prequantum bundle, the pull-back of any flat section by an element of the Heisenberg groups is still flat.  
So Equation (\ref{eq:Chern_Simons}) defines a flat section over the line $\frac{\la}{D} + ( \frac{\la}{D}  - \frac{\mu}{2}) \R$. To conclude it suffices to show that it has the right phase. 

The restriction $r$ sends $\mo ^{\ab} ( E_K)$ on $\pi ( \R \la)$, so the Chern-Simons invariant of any abelian representation $\rho $ is the class of $1 \in L_{t\la }$  where $t$ is any real such that $r ( \rho) = \pi ( t \la)$. The closure of $I_{\al, \be}$ in $\mo ( E_K)$ intersects $\mo^{\ab} ( E_K)$ at the two abelian representations with trace $2 \cos (\frac{\pi}{ab}\ell)$ where  $\ell = k^{-}_{\al, \be} $ and $k^+_{\al, \be} $ respectively. Since Equation (\ref{eq:Chern_Simons}) is equal to 1 for $\ell = k^-_{\al, \be} = t$, this proves the result for $\ell = k^-_{\al, \be}$. To deduce it from the other values of $\ell$, it suffices to compare the value of Equation (\ref{eq:Chern_Simons}) for $( \ell, t) = ( k^-_{\al, \be}, \ell' - k ^-_{\al, \be})  $ and $(\ell, t)  =( \ell' , 0)$.    
\end{proof}

\subsection{Reidemeister Torsion} 

The Reidemeister torsion of a knot exterior may be viewed a a density on the character manifold. It has been computed in \cite{dubois}. With the normalization of \cite{LJ1}, it is given on $I_{\al, \be}$ by  
\begin{xalignat*}{2} 
 \T = &
\frac{16}{a^2b^2}\sin^2 \Bigl( \frac{\pi\al}{a} \Bigr )\sin^2 \Bigl( \frac{\pi \be}{b} \Bigr ) 2^{3/2}\pi |r^* \mathrm{d}p| 
\end{xalignat*}
Here $p$ is the first linear coordinate on $E$ associated to the basis $(\mu, \la)$, so that $|dp|$ is  well-defined form on the quotient $\mo (\Si)$ except at the central representations.  
To compare with the asymptotic expansion of the knot state, let us evaluate the torsion on the tangent vector $D \mu - 2 \la$ to $r( I_{\al, \be})$. Since $|r^* d p | (D \mu - 2 \la) = D $, we have  
$$ \T ( D \mu -2 \la) = \frac{ 2^{13/2} \pi }{ab} \sin^2 \Bigl( \frac{\pi\al}{a} \Bigr )\sin^2 \Bigl( \frac{\pi \be}{b} \Bigr ).
$$

Recall that we denote by $\delta$ the half-form line and by $\varphi$ the isomorphim from $\delta ^2 $ to the canonical line (\ref{eq:canonical_line}). So for any $z \in \delta$ and $\rho \in I_{\al , \be}$, $\varphi (z^2)$ is a linear form on $E$, which defines by restriction a linear form on $ T_\rho I_{\al, \be} $. In this way any form in  $ T^*_\rho I_{\al, \be} \otimes \C$ has two square roots in $\delta$. The Reidemeister torsion $\T (\rho)$ being a density, it can be considered as a form in  $ T^*_\rho I_{\al, \be} \otimes \C$ well-defined up to sign, so it has four square roots in $\delta$. 

The square root which appears in Theorem \ref{theo:intro} is defined by
\begin{gather} \label{eq:racine_carre_torsion}
 \sqrt{\T_{\al, \be  }} :=  2^{13/4}  \Bigl( \frac{\pi}{ab} \Bigr)^{1/2}  ( -1 ) ^\al \sin \Bigl( \frac{\pi k_{\al, \be} ^{-}}{a} \Bigr )\sin \Bigl( \frac{\pi k^{-}_{\al, \be} }{b} \Bigr) \Om_{D \mu - 2 \la } 
\end{gather}
where $ \Om_{D \mu - 2 \la} \in \delta$ is such that $\Om_{D \mu - 2 \la} ( D \mu - 2 \la) =1$. Its sign is chosen so that Equation (\ref{eq:car_phi_0}) is satisfied. This expression defines indeed a square root of $\T (\rho)$ because by Lemma \ref{lem:morphism}, $k_{\al, \be}^{-} = \al$ mod $2a \Z$ and $k_{\al, \be}^{+} = \be$ mod $2b \Z$. 
4
\subsection{Proof of Theorem \ref{theo:intro}} 

First the knot state $Z_k (E_K) \in \Ga ( L^k_{\CS} \otimes \delta,  \mo ( \Si))$ considered in the introduction lifts to $Z_k \in \Ga (  L^k \otimes \delta , E)$, that is 
$$ Z_k (E_K) ( \pi ( x) ) = Z_k (x) , \qquad \forall x \in E$$ 
By Theorem \ref{theo:irreductible}, the knot state is a $O(k^{-\infty})$ at any point of $\pi ( A \setminus (\cup_{\ell \in \Z} I_\ell))$. From Theorem \ref{theo:abelian}, we deduce the asymptotic expansion over the image of $(\R \setminus D^{-1} \Z) \la$. 
It remains to consider the sets $\pi (I_{\ell})$ where the $I_\ell$'s are the intervals defined in (\ref{eq:intervalle_ell}). Since a fundamental domain of $(E , R \rtimes \Z_2)$ is $(-1,0) \mu + (0, \frac{1}{2}) \la $, it is sufficient to prove the result for $-1 \leqslant \ell \leqslant ab-1$. 

 Since for any $(\al, \be)$, $$ 1 \leqslant k_{\al, \be}^- < k^+_{\al, \be} \leqslant ab -1 ,$$ $r^{-1}( \pi ( I_\ell))$ is empty for $\ell = -1$ or $0$. So we have to prove that $Z_k ( \tau ) = O(k^{-\infty})$ for $\tau$ belonging to $\pi (I_{-1})$ or $\pi ( I_0)$. This follows from Theorem \ref{theo:irreductible}, because $\ga_{-1} (k) = 0$ and $\ga_{0} (k) = O(k^{-\infty})$. 

Then we proceed by induction over $\ell$. Recall that  
$r^{-1}( \pi ( I_{\ell}))$ is the union of the $I_{\ell , \al, \be }$ where $(\al, \be)$ runs over ${\mathcal{P}}_{\ell}$, cf. (\ref{eq:def_Iell_al_be}) and (\ref{eq:def_pl}).  For any $(\al , \be ) \in {\mathcal{P}}_\ell$, by Lemma \ref{lem:Chern_simons}, there exists a complex number $\ga_{\ell, \al, \be}(k)$ such that for any $t \in ( \ell, \ell +2)$, 
$$  \CS ^k ( \rho_{t, \al, \be} )  \frac{k^{3/4 } \mu_k } { 4 \pi ^{3/4}} \sqrt{\T_{\al, \be}} = \Bigl( \frac{k}{2\pi}\Bigr)^{1/4} i  \ga_{\ell, \al , \be} (k)  T^*_{ - \ell \lambda /D } t^k(x_t)  \otimes \Om_{D\mu - 2 \la} 
$$
where $x_t$ is the point of $I_{\ell}$ such that $\pi ( x_t) = \rho_{t, \al , \be}$. Furthermore, this complex number is given by 
\begin{gather} \label{eq:coef}
 \ga_{\ell, \al , \be}(k)  = \frac{\mu_k }{ i 2^{7/4} } \Bigl( \frac{k}{ \pi} \Bigr)^{1/2}
  e^{  \frac{i \pi k }{D}  ( \ell - k^-_{\al, \be} ) ( \ell + k^-_{\al, \be}) } \frac{\sqrt{ \T_{\al, \be}}}{\Om_{D \mu - 2 \la}} 
\end{gather}
We have to prove that
$$ \sum_{(\al, \be) \in {\mathcal{P}}_{\ell} } \ga_{\ell, \al , \be}(k) = \ga_{\ell} (k).$$
We will show that the left hand side satisfy the same recurrence relation as the $\ga_{\ell} (k)$'s.

There are three cases to consider according to $\ell$ is a multiple of $a$ or $b$, $\ell = k^{-}_{\al, \be}$ for some $(\al, \be)$ or $\ell = k^{+}_{\al, \be}$ for some $( \al, \be)$. Assume we are in the first case so that $ {\mathcal{P}}_{\ell-2, \al, \be} ={\mathcal{P}}_{\ell, \al ,\be} $. By equation (\ref{eq:coef}), 
\begin{gather}  \label{eq:coef_rec} 
 \ga_{\ell, \al , \be} (k) = e^{ \frac{2 i \pi }{D} 2k ( \ell-1) } \ga_{\ell -2 , \al , \be}(k) 
\end{gather}
We recover the recurrence relation of Theorem \ref{theo:irreductible} because $\be_{\ell}(k) = 0$ in this case. 

Assume now that $\ell = k^-_{\al', \be'} $ for some couple $(\al', \be')$ so that $ {\mathcal{P}}_\ell = {\mathcal{P}}_{\ell-2} \cup \{ (\al',\be') \}$.  Equation (\ref{eq:coef_rec}) is still satisfied for the couples $(\al, \be) \in {\mathcal{P}}_{\ell -2 } $. By a direct computation, we deduce from Equations (\ref{eq:coef_rec}) and (\ref{eq:racine_carre_torsion}) that
$$ \be_{\ell} (k) = \ga_{\ell, \al' , \be'} (k)  $$
which leads to the recurrence relation we look for. 

Finally, assume that $\ell = k^+_{\al', \be'} $ for some couple $(\al', \be')$ so that $ {\mathcal{P}}_\ell = {\mathcal{P}}_{\ell-2} \setminus \{ (\al',\be') \}$. Using that 
 \begin{gather*}  
(k^+_{\al, \be} + k^{-}_{\al, \be} ) ( k^+_{\al, \be} - k^{-}_{\al, \be} ) = 0 \mod 2D \Z
\end{gather*}
we obtain first that 
$$ e^{ \frac{2 i \pi }{D} 2k ( \ell-1) }  \ga_{\ell -2 , \al' , \be'}(k) = \frac{\mu_k }{ i 2^{7/4} } \Bigl( \frac{k}{ \pi} \Bigr)^{1/2}  \frac{\sqrt{ \T_{\al', \be'}}}{\Om_{D \mu - 2 \la}}  $$
Then using that 
\begin{gather*} 
 \sin \Bigl( \frac{\pi k_{\al, \be} ^{-}}{a} \Bigr )\sin \Bigl( \frac{\pi k^{-}_{\al, \be} }{b} \Bigr) +  \sin \Bigl( \frac{\pi k_{\al, \be} ^{+}}{a} \Bigr )\sin \Bigl( \frac{\pi k^{+}_{\al, \be} }{b} \Bigr) = 0 
\end{gather*}
  we obtain that
$$  e^{ \frac{2 i \pi }{D} 2k ( \ell-1) }  \ga_{\ell -2 , \al' , \be'}(k) + \be_{\ell} (k) = 0$$
which gives the conclusion. 

To end the proof, we have to consider the set $\pi ( D^{-1}(a\Z\cup b \Z) \la)$. For these representations, the result follows from Theorem \ref{theo:abelien_irreductible} and the previous considerations on the $\pi ( I_{\ell})$'s.

\section{Other related asymptotics expansions} 

Theorem \ref{theo:witten} is a consequence of Theorems \ref{theo:intro}, \ref{theo:irreductible}, \ref{theo:abelian} and \ref{theo:abelien_irreductible} as was explained in \cite{LJ2}, Section 4.2. We can deduce more general asymptotics for the WRT invariants of the Dehn filling of the torus knots where the attached solid torus contains a banded knot, Section 4.3 of \cite{LJ2}. In particular we obtain asymptotic expansion for evaluation of the colored Jones polynomials as follows.

Let $( \xi_{\ell})$ be the basis of $\Hilb_k$ determined by Equations (\ref{eq:rep_M_L}). By definition of the knot state, we have that 
$$ 
J_{\ell} \bigl( -  e^{i \pi / 2k} \bigr)  = \frac{ \sqrt k }{ \sin ( \pi / k ) } \langle Z_k , \xi_\ell \rangle 
$$ 
We can evaluate the scalar product $\langle Z_k , \xi_\ell \rangle $ by using some pairing formula, which gives the following theorem.  

\begin{theo} \label{sec:Jones_pol_das}
Let $\ell \in \Z$ and $I$ be a closed interval such that $I \subset (\frac{\ell}{D}, \frac{\ell + 1}{D}  )$. Then for any integers $k>0$ and $m$ such that $m / 2k \in I$, we have
\begin{gather*} 
 \langle Z_k , \xi_m \rangle = \sum_{(\al , \be) \in {\mathcal{P}}_{\ell-1} \cup{\mathcal{P}}_\ell } \frac{2 e^{ i p_{\al, \be} \frac{\pi}{2}} }{\sqrt {ab} } \sin \Bigl( \frac{\pi \al}{a} \Bigr )\sin \Bigl( \frac{\pi \be}{b} \Bigr)   e ^{ -i \pi k  \frac{( k^-_{\al, \be}  - \frac{m}{2k}D )^{2}}{D}} \\
 + e^{ i (\frac{\pi}{4} + p_0 \frac{\pi}{2}) } k^{-1/2} \frac{ \sin ( \pi \frac{  m }{k} ) }{\Delta_K( e ^{ 2 i \pi \frac{m}{k} })} +O ( k^{-1}) 
\end{gather*} 
where $p_0$, $p_{\al, \be}$ are integers and the $O(k^{-1})$ is uniform with respect to $k$ and $m$. 
\end{theo}

This is an application of Theorem 4.6 in \cite{LJ2} and Theorem 6.4 in \cite{LJ1}. 

\bibliography{biblio}

\end{document}